\documentclass[a4paper,12pt,leqno]{amsart}

\usepackage{amssymb, amsmath, amsthm}
\usepackage{amsfonts}
\usepackage{color}
\usepackage{mathrsfs}
\usepackage{graphicx}
\usepackage{comment}

\numberwithin{equation}{section}
\newcommand{\e}{\varepsilon}
\newcommand{\R}{\mathbb R}
\newcommand{\Z}{\mathbb Z}

\newcommand{\N}{{\mathbb N}}

\newcommand{\dv}{{\rm{div}}}

\newcommand{\wtts}{\overset{2,2}{\rightharpoonup}}
\newcommand{\stts}{\overset{2,2}{\to}}
\newcommand{\mfs}{\mathcal{M}}
\newcommand{\aop}{\mathcal{U}_\varepsilon}
\newcommand{\anop}{\mathcal{U}_{\varepsilon_n}}

\newcommand{\ufop}{\mathcal{T}_\varepsilon}
\newcommand{\ufnop}{\mathcal{T}_{\varepsilon_n}}
\newcommand{\vep}{\varepsilon}
\newcommand{\mx}{\tfrac{x}{\varepsilon}}
\newcommand{\mnx}{\tfrac{x}{\varepsilon_n}}
\newcommand{\mt}{\tfrac{t}{\varepsilon^r}}
\newcommand{\mnt}{\tfrac{t}{\varepsilon_n^r}}
\newcommand{\lam}{\lambda}

\newtheorem{thm}{Theorem}[section]
\newtheorem{lem}[thm]{Lemma}
\newtheorem{rmk}[thm]{Remark}
\newtheorem{prop}[thm]{Proposition}
\newtheorem{defi}[thm]{Definition}

\newtheorem{cor}[thm]{Corollary}

\allowdisplaybreaks[4]

\title{Corrector results for space-time homogenization of nonlinear diffusion}

\author{Tomoyuki Oka}
\address[Tomoyuki Oka]{Graduate School of Science, Tohoku University, Sendai 980-8579 Japan}
\email{tomoyuki.oka.q3@dc.tohoku.ac.jp}
\date{\today}


\begin{document}


\maketitle

\begin{abstract}
  The present paper concerns a space-time homogenization problem for nonlinear diffusion equations with periodically oscillating (in space and time) coefficients.
  Main results consist of corrector results (i.e., strong convergences of solutions with corrector terms) for gradients, diffusion fluxes and time-derivatives without assumptions for smoothness of coefficients.
  Proofs of the main results are based on the space-time version of the unfolding method, which is deeply concerned with the strong two-scale convergence theory.
\end{abstract}

\section{Introduction and main results}

\emph{Homogenization} is known as a method of asymptotic analysis for complex structures and systems.
Actually, it is often used to replace heterogeneous materials with a large number of microstructures, such as composite materials, with an equivalent homogeneous material; for instance, it is applied to models of heat conduction in composite materials.
Such models are often described as linear diffusion equations, and then their \emph{space-time homogenization} oscillating both in space and time has been investigated in various mathematical fields.

Space-time homogenization problems for linear diffusion equations are first studied by Bensoussan, Lions and Papanicolaou in \cite{ben} based on a method of \emph{asymptotic expansion}, and then various methods have been developed (see e.g.,~\cite{Ho} for \emph{two-scale convergence theory} and \cite{AAB17} for \emph{unfolding method}).
Furthermore, homogenization problems for parabolic equations have been studied not only for linear ones but also for nonlinear ones. 
In \cite{Ji00, NanRa01,Vis07}, doubly-nonlinear parabolic equations are treated, and moreover, as for degenerate $p$-Laplace parabolic equations, homogenization problems involving scale parameters (e.g., $r>0$ of $\mathrm{div}[A(\tfrac{x}{\e},\tfrac{t}{\e^r},\nabla u_\e)]$) are discussed in \cite{EfPa05, Wo10}.
On the other hand, \emph{porous medium types} are also treated in \cite{AO1, AO2}, and in particular, it is revealed that the difference between degeneracy and singularity of diffusion is deeply related to the representation of the so-called \emph{homogenized matrices} (see \cite[(ii) of Remark 1.5]{AO1}).

Throughout this paper, let $\Omega$ be a bounded domain of $\R^N$ with smooth boundary $\partial \Omega$, $a : \R^N \times \R_+ \to \R^{N\times N}$ be an $N \times N$ symmetric matrix field such that $a(y,s)$ is uniformly elliptic at each $(y,s) \in \R^N \times \R_+$, i.e., there exists $\lambda>0$ such that
\begin{equation}\label{eq:ellip}
 \lam|\xi|^2 \leq a(y,s) \xi \cdot \xi \leq |\xi|^2 \quad \mbox{ for all } \ \xi \in \R^N.
\end{equation}
Let us consider the following Cauchy-Dirichlet problem\/{\rm :}
\begin{equation}\label{eq:P}
  \left\{
    \begin{aligned}
    \partial_tu_{\e}&=\dv\left( a_\e\nabla |u_{\e}|^{p-1}u_{\e} \right)+f_{\e}
     &&\text{ in } \Omega\times (0,T), \\
    |u_{\e}|^{p-1}u_{\e}&=0 &&\text{ on } \partial\Omega\times (0,T), \\
    u_{\e}&=u^0  &&\text{ in } \Omega \times \{0\},
    \end{aligned}
    \right.
\end{equation}
where $\partial_t=\partial/\partial_t$, $u^0 \in H^{-1}(\Omega)$ and $f_\vep : (0,T) \to H^{-1}(\Omega)$ are given, $\e>0$, $a_\e=a(\mx,\mt)$ and $0<p, r<+\infty$.
Then \emph{weak solutions} $u_\e = u_\e(x,t) : \Omega \times (0,T) \to \R$ of \eqref{eq:P} are defined as follows.
\begin{defi}[Weak solution of \eqref{eq:P}]
A function $u_\e = u_\e(x,t) : \Omega \times (0,T) \to \R$ is called a \emph{weak solution} to \eqref{eq:P}, if the following conditions are all satisfied\/{\rm:}
\begin{itemize}
\rm
 \item[(i)] $u_{\e}\in W^{1,2}(0,T;H^{-1}(\Omega)) \cap L^{p+1}(\Omega\times(0,T))$, $|u_\vep|^{p-1}u_\e \in L^{2}(0,T;H^{1}_0(\Omega))$ and $u_{\e}(t)\to u^0$ strongly in $H^{-1}(\Omega)$ as $t\to 0_+$,
 \item[(ii)] it holds that
\begin{equation*}
\left\langle \partial_t u_{\e}(t),w\right\rangle_{H^1_0(\Omega)} + \int_{\Omega}a_\e\nabla (|u_\e|^{p-1}u_\e)(x,t)\cdot \nabla w(x)\ dx =\langle f_{\e}(t),w\rangle_{H^{1}_0(\Omega)}
\end{equation*}
for a.e.~$t\in (0,T)$ and all $w\in H^1_0(\Omega)$.
\end{itemize}
\end{defi}

\begin{rmk}
\rm
Nonlinear diffusion equations \eqref{eq:P} are
called \emph{porous medium equation} (PME) if $1<p<+\infty$ and \emph{fast diffusion equation} (FDE) if $0<p<1$
(see \cite{Va1,Va2} for details).
The well-posedness for \eqref{eq:P} can be obtained by \cite{Ak,AO1}.
\end{rmk}

We first recall the following fundamental result, which is obtained in \cite[Theorem 1.3]{AO1}.
\begin{thm}[Homogenization theorem for \eqref{eq:P}]\label{T:1}
Let $0 < p, r < +\infty$ and let $\e_n \to 0_+$ be an arbitrary sequence in $(0,+\infty)$.
In addition, suppose that
\begin{itemize}
\item $u^0\in L^2(\Omega) \cap L^{p+1}(\Omega)$,
\item $ f_{\e} \in W^{1,2}( 0,T ;H^{-1}(\Omega)) \cap L^1( 0,T;L^2(\Omega))$,
 \item $f_{\e_n} \to f$ weakly in $L^2( 0,T;H^{-1}(\Omega))$,
 \item $(f_{\e_n})$ is bounded in $L^1( 0,T;L^2(\Omega))$ if $p\in(0,1)$,
 \item $a\in [W^{1,1}(\R_+;L^{\infty}(\R^N))]^{N\times N}$ is $(\square \times J)$-periodic, i.e.,
 $$
 a(y+e_k, s+1)=a(y,s)\quad \text{ for a.e.~} (y,s)\in \R^N\times \R_+.
 $$
\end{itemize}
Here and henceforth, $\square:=(0,1)^N$ and $J := (0,1)$ stand for the unit cell and the unit interval, respectively, and $e_k$ denotes the $k$-th vector of the canonical basis of $\R^N$.
Let $u_{\e_n}$ be the unique weak solution to \eqref{eq:P}. Then there exist a subsequence of $(\e_n)$ and functions
\begin{align*}
u_{0}&\in W^{1,2}(0,T;H^{-1}(\Omega)) \cap L^{p+1}(\Omega\times(0,T)) 
\cap C_{\rm weak}([0,T];L^2(\Omega)),\\
z&\in L^{2}(\Omega\times (0,T) ;L^{2}(J;H^{1}_{\mathrm{per}}(\square)/\R))
\end{align*}
{\rm (}see Notation of Section {\rm\ref{S:Notation}} below{\rm )} such that $|u_{0}|^{p-1}u_{0} \in L^2(0,T;H^1_0(\Omega))$,
\begin{alignat*}{4}
|u_{\e_n}|^{p-1}u_{\e_n} &\to |u_{0}|^{p-1}u_{0} \quad &&\text{ weakly in }\ L^2(0,T;H^1_0(\Omega)), 
\\
u_{\e_n} &\to u_{0} \quad &&\text{ strongly in }\ L^\rho(0,T;L^{p+1}(\Omega))
\end{alignat*}
for any $\rho \in [1,+\infty)$ and
\begin{align*}
&a_{\e_n}\nabla |u_{\e_n}|^{p-1}u_{\e_n} \\
&\qquad\to j_{\rm hom}:=\left\langle a(\cdot,\cdot) \left(\nabla |u_{0}|^{p-1}u_0+\nabla_y z\right)\right\rangle_{y,s}\ \mbox{ weakly in } \  [L^2(\Omega \times  (0,T))]^N.
\end{align*}
Here and henceforth, $\nabla_y$ denotes gradients with respect to $y$ and $\langle \cdot\rangle_{y,s}$ denotes the mean over $\square\times J$, that is,
$$
\langle g \rangle_{y,s}=\int_0^1\int_{\square}g(y,s)\, dyds\quad \text{ for }\ g\in L^1(\square\times J).
$$
Moreover, the limit $u_0$ solves the weak form of the homogenized equation,
\begin{equation}\label{eq:P0}
\left\{
\begin{aligned}
&\langle \partial_tu_{0}(t),\phi\rangle_{H^1_0(\Omega)}+\int_{\Omega} j_{\rm hom}(x,t)\cdot\nabla\phi(x)\, dx  =\int_{\Omega} f(x,t)\phi(x)\, dx \ \mbox{ for } \, \phi\in H^1_0(\Omega),\\
&u_{0}( \cdot,0)=u^0 \ \mbox{ in } \Omega
\end{aligned}
 \right.
\end{equation}
for a.e.~$t \in (0,T)$.
\end{thm}

Furthermore, the homogenized diffusion flux $j_{\rm hom}$ is characterized as follows (see \cite{AO1,AO2} for the proof).
\begin{thm}[Characterization of $z$ and homogenized matrices
]\label{T:2}
Let $0 < p, r < +\infty$. In addition to all the assumptions of Theorem {\rm \ref{T:1}}, suppose that
\begin{equation*}
\begin{cases}
(f_{\e_n}) \mbox{ is bounded in } L^1( 0,T;L^{3-p}(\Omega))\ \text{ for } 0<p<2;\\
f_\e\equiv 0
\mbox{ and } u^0\ge 0
\mbox{ if }\ p\ge 2.
\end{cases}
\end{equation*}
In addition, assume that
\begin{equation*}
u^0 \in L^{3-p}(\Omega) \ \mbox{ if } \ p \in (0,1)\,{\rm ;} \
\log u^0 \in L^{1}_{\rm loc}(\Omega)\ \mbox{ if } \ p= 3\,{\rm ;} \
(u^0)^{3-p} \in L^{1}_{\rm loc}(\Omega)\ \mbox{ if } \ p\in (3,+\infty).
\end{equation*}
Let $u_0$ be a limit of weak solutions $(u_{\e_n})$ to \eqref{eq:P} as a sequence $\e_n \to 0_+$ and let $u_0$ be a weak solution of the homogenized equation \eqref{eq:P0}. 
Then
$z=z(x,t,y,s)$ appeared in Theorem {\rm \ref{T:1}} is represented as
\begin{equation}\label{eq:z}
z(x,t,y,s)=\sum_{k=1}^N(\partial_{x_k}v_0(x,t))\Phi_k(x,t,y,s),
\end{equation}
where 
$$
v_{0}=
\begin{cases}
|u_{0}|^{p-1}u_{0}\quad &\text{ if } p\in (0,2),\\
u_{0}^p\quad &\text{ if } p\in [2,+\infty)
\end{cases}
$$
and $\Phi_k\in L^{\infty}(\Omega\times I;L^2(J;H^1_{\mathrm{per}}(\square)/\R))$ is the corrector characterized as follows\/{\rm :}
\begin{description}
\item[\rm{(i)}] In case $0<r<2$, $\Phi_k=\Phi_k(y,s)$ is the unique weak solution to the cell-problem,
\begin{equation}\label{eq:CP1}
-\dv_y \left(a(y,s) \left[ \nabla_y\Phi_k(y,s)+e_{k}\right]\right)=0\ \text{ in }\ \square\times J.
\end{equation}

\item[\rm{(ii-FDE)}] In case $r=2$ and $p \in (0,1)$,
$\Phi_k=\Phi_k(x,t,y,s)$ solves the cell-problem,
\begin{equation}\label{eq:CP2}
\left\{
\begin{array}{ll}
\frac{1}{p}|u_0(x,t)|^{1-p}\partial_s\Phi_k(x,t,y,s)=\dv_y\left(a(y,s)\left[\nabla_y\Phi_k(x,t,y,s)+e_{k}\right]\right) \hspace{-1mm}&\mbox{in } \square\times J,\\
 \Phi_k(x,t,y,0)=\Phi_k(x,t,y,1) &\mbox{in } \square
\end{array}
\right.
\end{equation}
for each $(x,t) \in \Omega \times (0,T)$.
\item[\rm (ii-PME)] In case $r=2$ and $p\in(1,+\infty )$, $\Phi_k$ is given by
\begin{align}\label{eq:CP3}
\Phi_k(x,t,y,s)
= \begin{cases}
p |u_0(x,t)|^{p-1} \Psi_k(x,t,y,s) &\mbox{if } \ u_0(x,t) \neq 0,\\
0 &\mbox{if } \ u_0(x,t) = 0,
   \end{cases}
\end{align}
where $\Psi_k = \Psi_k(x,t,y,s)$ solves the cell-problem,
\begin{equation}\label{eq:CP4}
\left\{
\begin{array}{ll}
\partial_s\Psi_k(x,t,y,s)=\dv_y\left(a(y,s)\left[ p|u_0(x,t)|^{p-1}\nabla_y\Psi_k(x,t,y,s)+e_{k}\right]\right) \hspace{-1mm}&\mbox{in } \square\times J,\\
\Psi_k(x,t,y,0) = \Psi_k(x,t,y,1) &\mbox{in } \square
\end{array}
\right.
\end{equation}
for each $(x,t) \in [u_0\neq0] := \{(x,t) \in \Omega\times (0,T)\colon u_0(x,t) \neq 0\}$.
\item[\rm{(iii)}] In case $2<r<+\infty$,
$\Phi_k=\Phi_k(y)$ is the unique weak solution to the cell problem,
\begin{equation}\label{eq:CP5}
-\dv_y\biggl(\Bigl(\int_0^1a(y,s)\ ds\Bigl)[\nabla_y\Phi_k(y)+e_{k}]\biggl)=0\ \text{ in }\ \square.
\end{equation}
\end{description}
Moreover, the homogenized flux $j_{\rm hom}(x,t)$ can be written as
$$
j_{\rm hom}(x,t)=a_{\rm hom}(x,t)\nabla v_0(x,t).
$$
Here $a_{\rm hom}$ is the homogenized matrix given by
\begin{equation}\label{eq:ahom}
a_{\rm hom}(x,t) e_k = \int^1_0 \int_\square a(y,s) \left[ \nabla_y \Phi_k(x,t,y,s) + e_k \right] \, d y d s.
\end{equation}
\end{thm}

\begin{rmk}
\rm
It is noteworthy that the homogenized matrix \eqref{eq:ahom} is described in terms of solutions to cell problems, which have different forms depending on the log-ratio of the spatial and temporal periods of the coefficients. At a critical case (i.e., $r=2$), the cell problem turns out to be a parabolic equation in $(y,s)$ (as in linear diffusion) and also involves the limit of solutions, which is a function of $(x,t)$.
Thus $\Phi_k$ also depends on $(x,t)$, and hence, so does $a_{\rm hom}$.
On the other hand, as for $r\neq 2$, the cell problems are always elliptic equations, and then $\Phi_k$ is independent of $(x,t)$. Thus $a_{\rm hom}$ is a constant $N\times N$ matrix. Furthermore, the pair $(u_0, z)$ is uniquely determined, and hence, $(u_{\e_n})$ converges to $u_{0}$ without taking any subsequence.
\end{rmk}

\subsection{Main results}
The aim of the present paper is to establish the so-called corrector result for weak solutions to \eqref{eq:P}.
As mentioned in \cite{AO1, AG18, DM, Ho, NanRa01}, the gradient $\nabla u_{\e_n}^p$ does not converge to $\nabla u_0^p$ strongly in $[L^2(\Omega\times (0,T))]^N$ as $\e_n \to 0_+$.
However, by adding oscillating corrector terms, it follows that
\begin{align}
\nabla u_{\e_n}^p-\nabla u_0^p-\sum_{k=1}^N(\partial_{x_k}u_0^p)\nabla_y \Phi_k(x,t,\mnx,\mnt)\to 0
\ \text{ strongly in }\ [L^2(\Omega\times (0,T))]^N.
\label{eq:CR1}
\end{align}
To this end, we shall often need regularities\/{\rm :}$\nabla u_0^p\in [L^\sigma(\Omega\times (0,T))]^N$ and $\nabla_y\Phi_k\in [L^{\rho}(\Omega\times (0,T))]^N$ such that $\frac{1}{\sigma}+\frac{1}{\rho}=\frac{1}{2}$.
Then additional regularities of the coefficient $a(y,s)$ and given data may be required.
For instance, if $a\in [L^{\infty}(\R_+;C^{\alpha}(\R^{N}))]^{N\times N}$ for some $\alpha\in (0,1)$ and $(f_\e) $ is bounded in $L^2(0,T;L^{(p+1)/p}(\Omega))$, then $\nabla_y \Phi_k\in [L^{\infty}(\square\times J)]^N$ for $r\neq 2$ (see \cite{AGK} for details), and hence, \eqref{eq:CR1} holds true (see \cite[Theorem 1.6]{AO1}).
Moreover, at the critical case $r=2$, due to the dependence of $\Phi_k=\Phi_k(x,t,y,s)$, the smoothness of $a(y,s)$ is needed in \cite[Theorem 1.6]{AO1}.

The present paper provides another corrector result (introduced by \cite{CDG2}) without assumptions for the smoothness of $a(y,s)$ even for the critical case.

\begin{thm}[Corrector result for gradient]\label{T:cor}
Let $0 < p, r < +\infty$.
In addition to all the assumptions in Theorem {\rm\ref{T:2}}, suppose that
\begin{itemize}
 \item[] $f_{\e_n} \to f$ strongly in $L^2( 0,T;H^{-1}(\Omega))$ or weakly in $L^\sigma( 0,T;L^{(p+1)/p}(\Omega))$ 
\end{itemize}
for some $\sigma > 1$. Let $u_0$ be a limit of weak solutions $(u_{\e_n})$ to \eqref{eq:P} as a sequence $\e_n \to 0_+$ such that $u_0$ is a weak solution to \eqref{eq:P0}
and let $\Phi_{k}$ be the corrector given by \eqref{eq:CP1}--\eqref{eq:CP5}. 
Set
$$
v_{\e_n}:=
\begin{cases}
|u_{\e_n}|^{p-1}u_{\e_n}\quad &\text{ if } p\in (0,2),\\
u_{\e_n}^p\quad &\text{ if } p\in [2,+\infty),
\end{cases}\quad
v_{0}:=
\begin{cases}
|u_{0}|^{p-1}u_{0}\quad &\text{ if } p\in (0,2),\\
u_{0}^p\quad &\text{ if } p\in [2,+\infty).
\end{cases}
$$
Then it holds that
$$
\lim_{\e_n\to 0_+}\left\|\nabla v_{\e_n}-\nabla v_0-\sum_{k=1}^N\anop(\partial_{x_k}v_0)\,\anop(\nabla_y \Phi_k)
\right\|_{L^2(\Omega\times (0,T))}=0,
$$
where $\aop$ is the averaging operator
{\rm (}see Definition {\rm \ref{D:aop}} below{\rm )}.
\end{thm}

Furthermore, we have the following corollary (see also Remark \ref{R:nvanish})\/{\rm :}
\begin{cor}[Corrector results for diffusion flux and time-derivative]\label{C:cor}
Under the same assumptions as in Theorem {\rm\ref{T:cor}}, it holds that
\begin{align*}
&\bigg\| j_{\e_n} - j_{\rm hom}
- \Bigl[ a_{\e_n} \Bigl(
	\nabla v_0 + \sum_{k=1}^N \anop\left( \partial_{x_k} v_0 \right)\,\anop( \nabla_y\Phi_k )\Bigl) - j_{\rm hom} \Bigl] \bigg\|_{L^2(\Omega\times (0,T))}\nonumber
\end{align*}
$\to 0$ as $\e_n\to 0_+$, where $j_{\e_n} = a_{\e_n} \nabla v_{\e_n}$ and $j_{\rm hom}$ is the diffusion flux defined in Theorem {\rm \ref{T:1}}.  Moreover,  it holds that
\begin{align*}
&\bigg\| \partial_t v_{\e_n}^{1/p} - \partial_t v_0^{1/p}
- \mathrm{div} \Bigl[
a_{\e_n} \Bigl(
	\nabla v_0 + \sum_{k=1}^N \anop\left( \partial_{x_k} v_0 \right)\,\anop( \nabla_y\Phi_k )\Bigl)
- j_{\rm hom} \Bigl] \bigg\|_{L^2(0,T;H^{-1}(\Omega))}\nonumber
\end{align*}
$\to 0$ as $\e_n\to 0_+$.
\end{cor}

\subsection{Plan of the paper and notation}\label{S:Notation}
This paper is composed of three sections.
In the next section, we shall briefly review the relevant material on (weak and strong) \emph{space-time two-scale convergence} and \emph{space-time unfolding method}, and then we shall recall their relations. In the final section, we shall prove Theorem \ref{T:cor} and Corollary \ref{C:cor}.

\noindent
{\bf Notation.}\
Throughout this paper, we shall use the following notation\/:
\begin{itemize}
\item We simply write $I=(0,T)$ for $T>0$, and moreover, $\square=(0,1)^N$ and $J=(0,1)$ are the unit cell and interval, respectively. 
\item Define the set of smooth $\square$-periodic functions by
\begin{align*}
C^{\infty}_{\rm per}(\square)
&= \{w\in C^{\infty}(\R^N) \colon w(\cdot+e_k)=w(\cdot) \text{ in } \R^N \ \text{ for }\ 1\leq k \leq N\}.
\end{align*}
\item We also define $W^{1,q}_{\mathrm{per}}(\square)$ and $L^q_{\mathrm{per}}(\square)$ as closed subspaces of $W^{1,q}(\square)$ and $L^q(\square)$ by
$$
W^{1,q}_{\mathrm{per}}(\square) = \overline{C^\infty_{\rm per}(\square)}^{\|\cdot\|_{W^{1,q}(\square)}}, \quad L^q_{\mathrm{per}}(\square) = \overline{C^\infty_{\rm per}(\square)}^{\|\cdot\|_{L^q(\square)}} \simeq L^q(\square),
$$
respectively, for $1\leq q < +\infty$. In particular, set $H^1_{\rm per}(\square) := W^{1,2}_{\rm per}(\square)$. We shall simply write $L^q(\square)$ instead of $L^q_{\rm per}(\square)$, when no confusion can arise.
\item Denote by $C_{\rm weak}(\overline{I};X)$ the set of all weakly continuous functions defined on $\overline{I}$ with values in a normed space $X$.
Moreover, we write $X^N = X \times X \times \cdots \times X$ ($N$-product space), e.g., $[L^2(\Omega)]^N = L^2(\Omega;\R^N)$.
\item Define the set of Lebesgue measurable functions on $A\subset \R^{N+1}$ by $\mfs(A)$.
\item The unfolding operator $\ufop$ and
the averaging operator $\aop$ will be defined in Definitions \ref{D:ufop} and
\ref{D:aop}, respectively.
\end{itemize}

\section{Two-scale convergence theory and unfolding method}
Throughout this section, let $0<r<+\infty$ and $1< q< +\infty$, when no confusion can arise. Moreover, $q'$ denotes the H\"older conjugate of $q$, i.e.,~$1/q+1/q'=1$.
\subsection{Space-time two-scale convergence theory}
In this subsection, we shall recall the notion of \emph{space-time two-scale convergence} and briefly summarize its crucial properties without proofs (see e.g.,~\cite{Al, Ho, LNW,Ng, Vis06, Zh} for more details).

\begin{defi}[Weak and strong space-time two-scale convergence] \label{D:two}\
\begin{itemize}
\item[(i)]
A bounded sequence $(v_{\e})$ in $L^q(\Omega\times I)$ is said to \emph{weakly space-time two-scale converge} to a limit $v\in L^q(\Omega\times I\times \square\times J)$, if it holds that
\begin{align*}
\lim_{\e\to 0_+}
\int_0^T\int_{\Omega}v_{\e}(x,t)\Psi(x,t,\tfrac{x}{\e},\tfrac{t}{\e^r})\, dxdt
=
\int_0^T\int_{\Omega}\int_{0}^1\int_{\square}v(x,t,y,s)\Psi(x,t,y,s)\, dydsdxdt
\end{align*}
for any $\Psi(x,t,y,s)=\phi(x)\psi(t) b(y) c(s)$, $\phi\in C^{\infty}_{\rm c}(\Omega)$, $\psi\in C^{\infty}_{\rm c}(I)$, $b\in C^{\infty}_{\rm per}(\square)$ and $c\in C^{\infty}_{\rm per}(J)$, and it is denoted by
\[
v_{\e}\wtts v \quad \text{ in }\ L^q(\Omega\times I\times \square\times J).
\]
\item[(ii)]
A bounded sequence $(v_{\e})$ in $L^q(\Omega\times I)$ is said to \emph{strongly space-time two-scale converge} to a limit $v\in L^q(\Omega\times I\times \square\times J)$, if it holds that
\begin{align*}
\lim_{\e\to 0_+}
\int_0^T\int_{\Omega}v_{\e}(x,t)w_\e(x,t)\, dxdt
=
\int_0^T\int_{\Omega}\int_{0}^1\int_{\square}v(x,t,y,s)
w(x,t,y,s)\, dydsdxdt
\end{align*}
for any $w_\e\in L^{q'}(\Omega\times I)$ satisfying
$w_\e\wtts w$ in $L^{q'}(\Omega\times I\times \square\times J)$, and it is written as
\[
v_{\e}\stts v \quad \text{ in }\ L^q(\Omega\times I\times \square\times J).
\]
\end{itemize}
\end{defi}

As for the relation between (i) and (ii) in Definition \ref{D:two}, we have
\begin{prop}\label{P:scondi}
Let $v_\e\in L^{q}(\Omega\times I)$ be such that $v_\e\wtts v$ in $L^q(\Omega\times I\times \square\times J)$.
Then $v_\e\stts v$ in $L^q(\Omega\times I\times \square\times J)$, if and only if,
\begin{equation}\label{eq:chkstts}
\lim_{\e\to 0_+}\|v_\e\|_{L^q(\Omega\times I)}=\|v\|_{L^q(\Omega\times I\times \square\times J)}.
\end{equation}
\end{prop}

\begin{proof}
See \cite[Lemma 4.4]{Zh}.
\end{proof}

The following fact is useful to verify \eqref{eq:chkstts}.
\begin{prop}[Lower semicontinuity]\label{P:two-lsc}
Let $V_\e\wtts V$ in $[L^2(\Omega\times I\times \square\times J)]^N$. Then it holds that
\begin{equation}\label{eq:lsc}
\liminf_{\e\to 0_+}\int_0^T\int_{\Omega} |V_\e(x,t)|^2\, dxdt\ge
\int_0^T\int_\Omega\int_0^1\int_\square |V(x,t,y,s)|^2 dydsdxdt.
\end{equation}
Furthermore, for any symmetric matrices $W\in [L^{\infty}(\R^{N}\times\R_+)]^{N\times N}$ satisfying $(\square\times J)$-periodicity and uniform ellipticity {\rm(}e.g.,~\eqref{eq:ellip}{\rm)}, it follows that
\begin{align}\label{eq:matlsc}
&\liminf_{\e\to 0_+}\int_0^T\int_{\Omega} W(\mx,\mt)V_\e(x,t)\cdot V_\e(x,t)\, dxdt\\
&\quad \ge
\int_0^T\int_\Omega\int_0^1\int_{\square}
W(y,s)V(x,t,y,s)\cdot V(x,t,y,s)\, dydsdxdt.
\nonumber
\end{align}
\end{prop}

\begin{proof}
See {\cite[Section 7]{Zh}}.
\end{proof}

The following theorem is concerned with the weak space-time two-scale compactness.

\begin{prop}[Weak space-time two-scale compactness]\label{multicpt}
For any bounded sequence $(v_{\e})$ in $L^q(\Omega\times I)$, there exist a subsequence $(\e_n)$ of $(\e)$ such that $\e_n\to 0_+$ and a limit $v\in L^q(\Omega\times I\times \square\times J)$ such that
$$
v_{\e_n}\wtts v\quad \text{ in }\ L^q(\Omega\times I\times \square\times J).
$$
In addition, suppose that
$(v_{\e})$ be a bounded sequence in $L^{q}(I;W^{1,q}(\Omega))$ such that $v_{\vep_n} \to v$ strongly in $L^q(\Omega \times I)$ for some subsequence $(\vep_n)\to 0_+$ of $(\vep)$ and a limit $v\in L^q(I;W^{1,q}(\Omega))$.
Then there exist a {\rm(}not relaveled{\rm )} subsequence of $\vep_n\to 0_+$ and a function $z\in L^q(\Omega\times I;L^{q}(J; W^{1,q}_{\mathrm{per}}(\square)/\R))$ such that
\begin{align*}
  \nabla v_{\e_n}\wtts \nabla v+\nabla_y z\quad \text{ in }\ [L^q(\Omega\times I\times \square\times J)]^N.
\end{align*}
\end{prop}

\begin{proof}
See \cite[Theorems 2.3 and 3.1]{Ho}.
\end{proof}

\subsection{Space-time unfolding method}\label{unfolding}

The unfolding method was first introduced in \cite{CDG1} (see \cite{CDG2,CDG3} for more details), and then its space-time version was developed in \cite{AAB17}.
In this subsection, we briefly review the space-time unfolding method. Moreover, we shall prove a relation between the (weak and strong) space-time two-scale convergence and the (weak and strong) convergence of unfolded sequences for later use.

\begin{defi}[Space-time unfolding operator, {\cite[Definition 2.1]{AAB17}}]\label{D:ufop}
For $\e>0$, define the sets $\hat{\Omega}_{\e}\subset \Omega$ and $\hat{I}_{\e}\subset I$ by
\begin{align*}
\hat{\Omega}_{\e}&:=\text{\rm interior}\Bigl( \bigcup_{\xi\in \Xi_{\e}} \e(\xi+\overline{\square})\Bigl),\
\Xi_{\e}:=\{\xi\in \Z^N \colon \e(\xi+\square)\subset \Omega \},\\
\hat{I}_{\e}&:= \{ t\in I\colon \e^r\left( \lfloor\tfrac{t}{\e^r}\rfloor+1\right)\le T\},
\end{align*}
respectively. 
Here $\e(\xi+\overline{\square})$ denotes the cloesd $\e$-cell $[0,\e]^N$ with the origin at $\e\xi\in\e\Z^N$ and $\lfloor\cdot\rfloor$ denotes the floor function {\rm(}i.e.,~$\lfloor\cdot\rfloor$ denotes the integer part of $\cdot${\rm)}.
Set $\Lambda_\e:=(\Omega\times I)\setminus (\hat{\Omega}_{\e}\times \hat{I}_{\e})$.
For $\e>0$, the \emph{space-time unfolding operator} $\ufop: \mfs(\hat{\Omega}_{\e}\times \hat{I}_{\e})\to \mfs(\hat{\Omega}_{\e}\times \hat{I}_{\e}\times\square\times J)$ is defined by
\begin{align*}
\ufop(w)(x,t,y,s)=
\begin{cases}
w(\e\lfloor\mx\rfloor+\e y, \e^r\lfloor\mt\rfloor+\e^r s)\quad  &\text{for a.e.~}(x,t,y,s )\in \hat{\Omega}_{\e}\times \hat{I}_{\e}\times \square\times J,\\
0 &\text{for a.e.~}(x,t)\in
\Lambda_{\e}\times \square\times J,
\end{cases}
\end{align*}
for $w\in\mfs(\hat{\Omega}_{\e}\times \hat{I}_{\e})$.
Moreover, the unfolding operator {\rm (}still denoted by $\ufop${\rm )} can be defined analogously for $W\in \mfs(\hat{\Omega}_{\e}\times \hat{I}_{\e})^N=\mfs(\hat{\Omega}_{\e}\times \hat{I}_{\e};\R^N)$.
\end{defi}

\begin{rmk}\label{R:uf}
\rm
We readily get the following (i)--(iii)\/{\rm :}
\begin{itemize}
\item[(i)]$\ufop(v+w)=\ufop(v)+\ufop(w)$ for all $v,w\in \mfs(\hat{\Omega}_{\e}\times \hat{I}_{\e})$.
\item[(ii)]$\ufop(vw)=\ufop(v)\ufop(w)$ for all $v,w\in \mfs(\hat{\Omega}_{\e}\times \hat{I}_{\e})$.
\item[(iii)]$|\ufop(w)|^\gamma=\ufop(|w|^\gamma)$ for all $w\in \mfs(\hat{\Omega}_{\e}\times \hat{I}_{\e})$ and all $\gamma\in \R$.
\end{itemize}
\end{rmk}

\begin{rmk}\label{R:cng}
\rm
Note that the following property holds\/{\rm :}
$$
\iint_{\hat{\Omega}_\e\times \hat{I}_\e} w(x,t)\, dxdt
=
\int_0^T\int_\Omega\int_0^1\int_\square
\ufop(w)(x,t,y,s)\, dydsdxdt
$$
for $w\in \mfs(\hat{\Omega}_\e\times \hat{I}_\e)$.
Indeed, in Definition \ref{D:ufop}, $\hat{I}_\e$ can be represented as
$$
\hat{I}_\e=\text{interior}\Bigl(\bigcup_{\zeta\in \Theta_{\e}} \e^r(\zeta+\overline{J})\Bigl),\
\Theta_{\e}:=\{\zeta\in \N\cup\{0\} \colon \e^r(\zeta+J)\subset I \}.
$$
Hence it follows that, for all $w\in L^1(\hat{\Omega}_\e\times \hat{I}_\e)$,
\begin{align*}
\iint_{\hat{\Omega}_\e\times \hat{I}_\e} w(x,t)\, dxdt
&=
\sum_{\zeta\in \Theta_\e}\sum_{\xi\in \Xi_\e}
\int_{\e^r(\zeta+J)}\int_{\e(\xi+\square)} w(x,t)\, dxdt\\
&=
\e^{N+r}
\sum_{\zeta\in \Theta_\e}\sum_{\xi\in \Xi_\e}
\int_0^1\int_{\square}
w(\e\xi+\e y,\e^r\zeta+\e^r s)\, dyds\\
&=
\sum_{\zeta\in \Theta_\e}\sum_{\xi\in \Xi_\e}
\iint_{\e^r(\zeta+J)\times J}
\iint_{\e(\xi+\square)\times \square}
w(\e\xi+\e y,\e^r\zeta+\e^r s)\, dydsdxdt\\
&=
\int_0^T\int_\Omega\int_0^1\int_\square
\ufop(w)(x,t,y,s)\, dydsdxdt.
\end{align*}
\end{rmk}

The following lemma is concerned with the relation between
the weak convergence for unfolded sequences and the weak space-time two-scale convergence.
\begin{prop}[cf.~{\cite[Proposition 2.14]{CDG2}}]\label{weak-wtts}
Let $(v_\e)$ be a bounded sequence in $L^q(\Omega\times I)$ and $v\in L^q(\Omega\times I\times \square\times J)$. Then the following assertions are equivalent\/{\rm :}
\begin{itemize}
\item[\rm (i)] $\ufop(v_\e)\to v$ weakly in $L^q(\Omega\times I\times \square\times J)$,
\item[\rm (ii)] $v_\e\wtts v$ in $L^q(\Omega\times I\times \square\times J)$.
\end{itemize}
\end{prop}

\begin{proof}
For any $\phi\in C^{\infty}_{\rm c}(\Omega)$, $\psi\in C^{\infty}_{\rm c}(I)$, $b\in C^{\infty}_{\rm per}(\square)$ and $c\in C^{\infty}_{\rm per}(J)$, set
$\Psi(x,t,y,s)=\phi(x)\psi(t) b(y)c(s)$ and $\Psi_\e(x,t)=\phi(x)\psi(t) b(\mx) c(\mt)$.
Let $\e>0$ be small enough such that $\Psi_\e=0$ on $\Lambda_\e$.
From the $(\square\times J)$-periodicity of $\Psi$, it holds that
\begin{align*}
\ufop\left(\Psi_\e\right)(x,t,y,s)
&=\Psi\left(\e\lfloor\mx\rfloor+\e y,\e^r\lfloor\mt\rfloor+\e^rs, \lfloor\mx\rfloor+y ,\lfloor\mt\rfloor+s
\right)\\
&=\Psi\left(\e\lfloor\mx\rfloor+\e y, \e^r\lfloor\mt\rfloor+\e^rs, y, s \right).
\end{align*}
Hence the uniformly continuity
of $\Psi$ on $\Omega\times I\times \square\times J$ ensures that
\begin{equation}\label{unifcontitest}
\lim_{\e\to 0_+}|\ufop(\Psi_\e)(x,t,y,s)-\Psi(x,t,y,s)|= 0.
\end{equation}

Assume (i) first. Then we derive by Remarks \ref{R:uf} and \ref{R:cng} and \eqref{unifcontitest} that
\begin{align*}
\int_0^T\int_\Omega v_\e(x,t)\Psi_\e(x,t)\, dxdt
&=
\iint_{\hat{\Omega}_\e\times \hat{I}_\e} v_\e(x,t)\Psi_\e(x,t)\, dxdt\\
&=
\int_0^T\int_\Omega\int_0^1\int_\square
\ufop(v_\e)(x,t,y,s)\ufop(\Psi_\e)(x,t,y,s)\, dydsdxdt\\
&\to
\int_0^T\int_\Omega\int_0^1\int_\square
v(x,t,y,s)\Psi(x,t,y,s)\, dydsdxdt
\end{align*}
as $\e\to 0_+$. Thus (ii) follows.

Next, suppose that (ii) is fulfilled. Noting by Remarks \ref{R:uf} and \ref{R:cng} that
$$
\|\ufop(v_\e)\|_{L^q(\Omega\times I\times \square\times J)}^q
=
\|\ufop(|v_\e|^q)\|_{L^1(\Omega\times I\times \square\times J)}
=
\iint_{\hat{\Omega}_\e\times \hat{I}_\e}|v_\e(x,t)|^q\, dxdt\le C,
$$
one can assume that
$$
\ufnop(v_{\e_n})\to  \tilde{v}\quad \text{ weakly in } L^q(\Omega\times I\times \square\times J)
$$
for some $\tilde{v}\in L^q(\Omega\times I\times \square\times J)$ and $\e_n\to 0_+$.
Then we conclude by Remarks \ref{R:uf} and \ref{R:cng} and \eqref{unifcontitest} that
\begin{align*}
&\int_0^T\int_\Omega\int_0^1\int_\square \tilde{v}(x,t,y,s)\Psi(x,t,y,s)\, dydsdxdt\\
&\qquad=\lim_{\e_n\to 0_+}
\int_0^T\int_\Omega\int_0^1\int_\square \ufnop(v_{\e_n})(x,t,y,s)\ufnop(\Psi_{\e_n})(x,t,y,s)\, dydsdxdt\\
&\qquad=\lim_{\e_n\to 0_+}\Bigl(
\int_0^T\int_{\Omega} v_{\e_n}(x,t)\Psi_{\e_n}(x,t)\, dxdt
-\underbrace{\iint_{\Lambda_{\e_n}} v_{\e_n}(x,t)\Psi_{\e_n}(x,t)\, dxdt}_{= 0\ \text{ by } \Psi_{\e_n}=0  \text{ on } \Lambda_{\e_n}}\Bigl) 
\\
&\qquad=
\int_0^T\int_\Omega\int_0^1\int_\square v(x,t,y,s)\Psi(x,t,y,s)\, dydsdxdt,
\end{align*}
which together with the arbitrariness of the test function $\Psi$
yields (i). This completes the proof.
\end{proof}

Combining Proposition \ref{weak-wtts} with Proposition \ref{multicpt}, we obtain
\begin{prop}[Weak compactness for space-time unfolded sequences]\label{ufcpt}
For any bounded sequence $(v_{\e})$ in $L^q(\Omega\times I)$, there exist a subsequence  $\e_n \to 0_+$  of $(\e)$  and a function $v$\,$\in$\, $L^q(\Omega\times I\times \square\times J ) $ such that
\[
  \mathcal{T}_{\e_n}(v_{\e_n})\to v \quad \text{ weakly in } L^q(\Omega\times I\times \square\times J) .
\]
In addition, assume that $(v_{\e})$ is bounded in $L^{q}(I;W^{1,q}(\Omega))$ and $ v_{\vep_n} \to  v$ strongly in $L^q(\Omega \times I)$ for a limit $v\in L^q(I;W^{1,q}(\Omega))$. Then there exist a {\rm(}not relaveled{\rm )}  subsequence of $\vep_n\to 0_+$ and a function $z\in L^q(\Omega\times I;L^{q}(J; W^{1,q}_{\mathrm{per}}(\square)/\R))$ such that
\begin{align*}
\ufnop(\nabla v_{\e_n}) \to \nabla v+\nabla_y z\quad \text{ weakly in }\ [L^q(\Omega\times I\times \square\times J)]^N.
\end{align*}
\end{prop}

We next introduce the space-time averaging operator.
\begin{defi}[Space-time averaging operator]\label{D:aop}
Under the same assumption as in Definition {\rm\ref{D:ufop}}, the \emph{space-time averaging operator} $\,\aop\colon L^q(\Omega\times I\times \square\times J)\to L^q(\Omega\times I)$ is defined as follows\/{\rm :}
\begin{align*}
\aop(\Psi)(x,t)=
\begin{cases}
\displaystyle
\int_0^1\int_\square\Psi(\e\lfloor \mx\rfloor+\e\sigma,\e^r\lfloor \mt\rfloor+\e^r\rho, \{\mx\},\{\mt\})\, d\sigma d\rho  &\text{for a.e.~}(x,t)\in \hat{\Omega}_{\e}\times \hat{I}_{\e},\vspace{3mm}\\
0 &\text{for a.e.~}(x,t)\in \Lambda_{\e}.
\end{cases}
\end{align*}
Here $\{\cdot\}$ denotes the  fraction part of $\cdot$ {\rm(}i.e., ~$\{\cdot\}:=\cdot-\lfloor \cdot\rfloor${\rm)}.
\end{defi}

\begin{rmk}
\rm 
Note that the space-time averaging operator $\aop$ is also linear by the linearity of the Lebesgue integral.
\end{rmk}

\begin{rmk}[Adjoint of space-time unfolding operator]\label{R:ad}
\rm
In an analogy way of \cite[Section 2.2]{CDG2},
the space-time averaging operator $\,\aop$ can be regarded as the adjoint of $\ufop$. Hence it holds that, for any $\Psi\in L^{q'}(\Omega\times I\times \square\times J)$ and $\psi\in L^q(\Omega\times I)$,
$$
\int_0^T\int_\Omega \aop(\Psi)(x,t)\psi(x,t)\, dxdt=
\int_0^T\int_\Omega\int_0^1\int_\square \Psi(x,t,y,s)\ufop(\psi)(x,t,y,s)\, dydsdxdt.
$$
Furthermore, $\aop$ is almost a left inverse of $\ufop$, that is,
$$
\aop(\ufop(\psi))(x,t)=
\begin{cases}
\psi(x,t)  &\text{for a.e.~}(x,t)\in \hat{\Omega}_{\e}\times \hat{I}_{\e},\vspace{3mm}\\
0 &\text{for a.e.~}(x,t)\in \Lambda_{\e},
\end{cases}
$$
since it folllows from Remarks \ref{R:uf} and \ref{R:cng} that
\begin{align*}
\iint_{\hat{\Omega}_\e\times \hat{I}_\e}\aop(\ufop(\psi))(x,t)\phi(x,t)\, dxdt
&=\int_0^T\int_{\Omega}\int_0^1\int_{\square}
\ufop(\psi)(x,t,y,s)\ufop(\phi)(x,t,y,s)\, dydsdxdt\\
&=\iint_{\hat{\Omega}_\e\times \hat{I}_\e}\psi(x,t)\phi(x,t)\, dxdt
\quad \text{ for all } \phi\in C^{\infty}_{\rm c}(\hat{\Omega}_\e\times \hat{I}_\e).
\end{align*}
\end{rmk}

As for the strong convergence of unfolded sequences and the strong two-scale convergence, the following relation holds.

\begin{prop}\label{equiv-stts}
Let $(v_\e)$ be bounded in $L^q(\Omega\times I)$
and let $v\in L^q(\Omega\times I\times \square\times J)$.
Then the following {\rm(i)}--{\rm (iii)} are equivalent\/{\rm :}
\begin{itemize}
\item[\rm (i)] $\ufop(v_\e)\to v$ strongly in $L^q(\Omega\times I\times \square\times J)$ and $\lim_{\e\to 0_+}\iint_{\Lambda_{\e}}|v_\e(x,t)|^q\, dxdt= 0$,
\item[\rm (ii)]
$v_\e-\aop(v)\to 0$ strongly in $L^q(\Omega\times I)$,
\item[\rm (iii)]  $v_\e\stts v$ in $L^q(\Omega\times I\times \square\times J)$.
\end{itemize}
\end{prop}

\begin{proof}
Since the relation (i) $\Leftrightarrow$ (ii) follows from \cite[(iv) of Proposition 2.18]{CDG2}, we shall prove the relation (i) $\Leftrightarrow$ (iii) only. Note by Remarks \ref{R:uf} and \ref{R:cng} that
\begin{align}
\lefteqn{\int_0^T\int_{\Omega} |v_\e(x,t)|^q\, dxdt}\label{keyeq}\\
&\quad =
\int_0^T\int_{\Omega}\int_0^1\int_{\square} |\ufop(v_\e)(x,t,y,s)|^q
\, dydsdxdt + \iint_{\Lambda_\e}|v_\e(x,t)|^q\, dxdt.\nonumber
\end{align}
If (i) holds, then we have
$$
\lim_{\e\to 0_+} \int_0^T\int_{\Omega} |v_\e(x,t)|^q\, dxdt=\int_0^T\int_{\Omega}\int_0^1\int_{\square} |v(x,t,y,s)|^q\, dydsdxdt,
$$
which along with Proposition \ref{P:scondi} yields (iii).

Conversely, if (iii) holds, then it follows that
\begin{equation}\label{errorchk}
\iint_{\Lambda_\e}|v_\e(x,t)|^q\, dxdt
\to  0 \quad  \text{ as } \e\to 0_+.
\end{equation}
Indeed, by \eqref{keyeq}, it holds that
$$
\int_0^T\int_{\Omega} |v_\e(x,t)|^q\, dxdt
\ge
\int_0^T\int_{\Omega}\int_0^1\int_{\square} |\ufop(v_\e)(x,t,y,s)|^q
\, dydsdxdt.
$$
Combining Proposition \ref{P:scondi} with Proposition \ref{weak-wtts} and the weak lower semicontinuity of norm, one can derive by  that
\begin{align}
\label{keyeq2}
\lim_{\e\to 0_+}\int_0^T\int_{\Omega} |v_\e(x,t)|^q\, dxdt
&=
\lim_{\e\to 0_+}\int_0^T\int_{\Omega}\int_0^1\int_{\square}
|\ufop(v_{\e})(x,t,y,s)|^q
\, dydsdxdt\\
&=
\int_0^T\int_{\Omega}\int_0^1\int_{\square} |v(x,t,y,s)|^q\, dydsdxdt,\nonumber
\end{align}
which along with \eqref{keyeq} yields \eqref{errorchk}.
Furthemore, by \eqref{keyeq2} and the uniform convexity of $L^q(\Omega\times I\times \square\times J)$, we obtain (i). This completes the proof.
\end{proof}

\section{Proof of Main results}
Before proving Theorem \ref{T:cor} and Corollary \ref{C:cor}, 
recall by Theorem \ref{T:1} and Proposition \ref{multicpt} that
\begin{equation}\label{eq:grawtts}
a_{\e_n}\nabla v_{\e_n}\wtts a(y,s)(\nabla v_0+\nabla_y z)\quad \text{ in }\ [L^2(\Omega\times I\times\square\times J)]^N,
\end{equation}
where $z=\sum_{k=1}^N(\partial_{x_k}v_0)\Phi_k$ and $\Phi_k$ is the corrector given by \eqref{eq:CP1}--\eqref{eq:CP5}.
\subsection{Proof of Theorem \ref{T:cor}}
We first observe that
\begin{align*}
\lefteqn{
\bigl\|\nabla v_{\e_n}-\nabla v_0-\sum_{k=1}^N\anop(\partial_{x_k}v_0)\,\anop(\nabla_y \Phi_k) \bigl\|_{L^2(\Omega\times I)}}\\
&\quad \le
\bigl\|\nabla v_{\e_n}-\anop(\nabla v_{0})-\anop(\nabla_y z)\bigl\|_{L^2(\Omega\times I)}\\
&\qquad+
\bigl\|\anop(\nabla v_{0})-\nabla v_{0}\bigl\|_{L^2(\Omega\times I)}
+
\bigl\|\anop(\nabla_y z)-\sum_{k=1}^N\anop(\partial_{x_k}v_0)\,\anop(\nabla_y \Phi_k) \bigl\|_{L^2(\Omega\times I)}\\
&\quad =: I_1^{\e_n}+I_2^{\e_n}+I_3^{\e_n}.
\end{align*}
Then we shall estimate the terms $I_1^{\e_n}$, $I_2^{\e_n}$ and $I_3^{\e_n}$ below.

To prove
\begin{equation}
I_1^{\e_n}\to 0\quad \text{ as }\ \e_n\to 0_+,\label{I1est}
\end{equation}
we need the following
\begin{lem}\label{grastts}
Under the same assumption as in Theorem {\rm\ref{T:cor}}, it holds that
$$
\nabla v_{\e_n}\stts \nabla v_{0}+\nabla_y z\quad \text{ in } [L^2(\Omega\times I\times\square\times J)]^N.
$$
\end{lem}

\begin{proof}[Proof of Lemma \ref{grastts}]
Noting by \cite[Lemma 6.1]{AO1} that
\begin{align*}
  v_{\e_n}(t)^{1/p}\to v_{ 0}(t)^{1/p}\quad \text{ weakly in } L^{p+1}(\Omega)\quad \text{ \underline{for all} }\  t\in \overline{I},
\end{align*}
we deduce that
\begin{align*}
\lefteqn{\limsup_{\e_n\to 0_+}\int_0^T\int_{\Omega}a_{\e_n}\nabla v_{\e_n}\cdot \nabla v_{\e_n}\, dxdt}\\
&\stackrel{\eqref{eq:P}}{=}
\limsup_{\e_n\to 0_+}\left(\int_0^T\int_{\Omega}f_{\e_n}v_{\e_n}\, dxdt-\int_0^T \langle \partial_t v_{\e_n}^{1/p}, v_{\e_n}\rangle_{H^1_0(\Omega)}\, dt\right)\\
&= \limsup_{\e_n\to 0_+}\left[\int_0^T\int_{\Omega}f_{\e_n}v_{\e_n}\, dxdt-\frac{1}{p+1}\left(\|v_{\e_n}^{1/p}(T)\|_{L^{p+1}(\Omega)}^{p+1} - \|u^0\|_{L^{p+1}(\Omega)}^{p+1}\right)\right]\\
&\le \int_0^T\int_{\Omega}fv_{0}\, dxdt - \frac{1}{p+1}\left(\|v_0^{1/p}(T)\|_{L^{p+1}(\Omega)}^{p+1} - \|u^0\|_{L^{p+1}(\Omega)}^{p+1}\right)\\
&= \int_0^T\int_{\Omega}fv_{0}\, dxdt-\int_0^T \langle \partial_t v_{0}^{1/p}, v_{0}\rangle_{H^1_0(\Omega)}\, dt\\
&\stackrel{\eqref{eq:P0}, \eqref{eq:z},\eqref{eq:ahom}}{=}\int_0^T\int_{\Omega} a_{\rm hom}(x,t)\nabla v_0\cdot \nabla v_0\, dxdt.
\end{align*}

On the other hand, \eqref{eq:matlsc} of Propositiion \ref{P:two-lsc}, \eqref{eq:grawtts} and \eqref{eq:ahom} yield that
\begin{align*}
\lefteqn{\liminf_{\e_n\to 0_+}
\int_0^T\int_{\Omega} a_{\e_n}\nabla v_{\e_n}\cdot \nabla v_{\e_n}\, dxdt}\\
&\quad\ge
\int_0^T\int_{\Omega}\int_0^1\int_{\square} a(y,s)(\nabla v_{0}+\nabla_y z)\cdot (\nabla v_{0}+\nabla_y z)\, dydsdxdt\\
&\quad =
\int_0^T\int_{\Omega}a_{\rm hom}(x,t)\nabla v_{0}\cdot \nabla v_{0}\, dxdt\\
&\qquad +
\underbrace{\int_0^T\int_{\Omega}\int_0^1\int_{\square} a(y,s)\left( \nabla v_{0}+\nabla_y z\right)\cdot \nabla_y z\, dydsdxdt.}_{=0\ \text{ by \eqref{eq:CP1}-\eqref{eq:CP5} and the $J$-periodicity of $\Phi_k$ for $r=2$} }
\end{align*}
Thus it follows that
\begin{align}\label{key1}
\lim_{\e_n\to 0_+}\int_0^T\int_{\Omega}a_{\e_n}\nabla v_{\e_n}\cdot \nabla v_{\e_n}\, dxdt=
\int_0^T\int_{\Omega}a_{\rm hom}(x,t)\nabla v_{0}\cdot \nabla v_{0}\, dxdt.
\end{align}

Now, let $\mathbb{I}\in \R^{N\times N}$ be a unit matrix and let $\gamma>0$ be such that $(a(y,s)-\gamma \mathbb{I})\xi\cdot \xi\ge \tilde{\lambda}|\xi|^2 $ for all $\xi\in \R^N$ and some $\tilde{\lambda}>0$.
Then we infer by \eqref{eq:matlsc} of Proposition \ref{P:two-lsc}, \eqref{eq:grawtts} and \eqref{eq:ahom} that
\begin{align*}
\lefteqn{
\liminf_{\e_n\to 0_+}\int_0^T\int_{\Omega}\left(a_{\e_n}-\gamma \mathbb{I}\right)\nabla v_{\e_n}\cdot \nabla v_{\e_n}\, dxdt }\\
&\ge
\int_0^T\int_{\Omega}\int_0^1\int_{\square} \left(a(y,s)-\gamma \mathbb{I}\right)\left( \nabla v_{0}+\nabla_y z\right)\cdot \left( \nabla v_{0}+\nabla_y z\right)\, dydsdxdt\\
&=
\int_0^T\int_{\Omega}a_{\rm hom}(x,t)\nabla v_{0}\cdot \nabla v_{0}\, dxdt\\
&\quad +
\underbrace{\int_0^T\int_{\Omega}\int_0^1\int_{\square} a(y,s)\left( \nabla v_{0}+\nabla_y z\right)\cdot \nabla_y z\, dydsdxdt}_{=0\ \text{ by \eqref{eq:CP1}-\eqref{eq:CP5} and the $J$-periodicity of $\Phi_k$ for $r=2$} }\\
&\quad -
\int_0^T\int_{\Omega}\int_0^1\int_{\square} \gamma \mathbb{I}\left( \nabla v_{0}+\nabla_y z\right)\cdot \left( \nabla v_{0}+\nabla_y z\right)\, dydsdxdt,
\end{align*}
and hence, \eqref{key1} ensures that
\begin{align*}
\lefteqn{
-\limsup_{\e_n\to 0_+}\int_0^T\int_{\Omega}\gamma \mathbb{I}\nabla v_{\e_n}\cdot \nabla v_{\e_n}\, dxdt }\\
&\qquad \ge
-
\int_0^T\int_{\Omega}\int_0^1\int_{\square} \gamma \mathbb{I}\left( \nabla v_{0}+\nabla_y z\right)\cdot \left( \nabla v_{0}+\nabla_y z\right)\, dydsdxdt,
\end{align*}
which together with \eqref{eq:lsc} of Proposition \ref{P:two-lsc} yields that
\begin{equation}\label{chk_stts}
\lim_{\e_n\to 0_+}\|\nabla v_{\e_n}\|_{L^2(\Omega\times I)}^2
=
\|\nabla v_0+\nabla_y z\|_{L^2(\Omega\times I\times \square\times J)}^2.
\end{equation}
Thanks to \eqref{chk_stts} and Proposition \ref{P:scondi}, we get the assertion.
\end{proof}
By virtue of Lemma \ref{grastts}, \eqref{I1est} follows from the implication (iii) $\Rightarrow$ (ii) of Proposition \ref{equiv-stts}.

We next claim that
\begin{equation}\label{I2est}
I_2^{\e_n}\to 0\quad \text{ as }\ \e_n\to 0_+.
\end{equation}
This also follows from the implication (iii) $\Rightarrow$ (ii) of Proposition \ref{equiv-stts}. Indeed, since $\nabla v_0$ is independent of $(y,s)\in \square\times J$, and in particular, $\nabla v_0$ strongly space-time two-scale converges to itself in $[L^2(\Omega\times I\times \square\times J)]^N$, one can get the assertion.

We finally show that
\begin{equation}\label{I3est}
I_3^{\e_n}\to 0\quad \text{ as }\ \e_n\to 0_+.
\end{equation}
It suffices to prove that
\begin{equation}\label{final}
\anop((\partial_{x_k}v_{0})\nabla_y \Phi_k)
-\anop(\partial_{x_k}v_{0})\,
\anop(\nabla_y \Phi_k)\to 0\quad \text{ strongly in }\ [L^2(\Omega\times I)]^N.
\end{equation}
To this end, we use the following fact\/{\rm :}
\begin{equation}\label{phireg}
\nabla_y \Phi_k\in
[L^{ \infty}( \Omega\times I; L^2(\square\times J))]^N,
\end{equation}
which plays a crucial role at the critical case $r=2$ (see \cite[Appendix]{AO1} for the proof).
Since $\partial_{x_k}v_{0}$ is independent of $(y,s)\in \square\times J$,
noting that, for any $(\xi,\zeta)\in \Xi_{\e_n}\times \Theta_{\e_n}$ (see \S \ref{unfolding} for notation), $\aop(\partial_{x_k}v_{0})$ can be regarded as a constant in $\e_n(\xi+\square)\times \e_n^r(\zeta+J)$,
we derive by Remark \ref{R:ad}  and H\"{o}lder's inequality that
\begin{align*}
\lefteqn{
\left\|\anop((\partial_{x_k}v_{0})\nabla_y \Phi_k)
-\anop(\partial_{x_k}v_{0})\,
\anop(\nabla_y \Phi_k)\right\|_{L^2(\Omega\times I)}^2
}\\
&\quad=
\sum_{\zeta\in \Theta_\e}\sum_{\xi\in \Xi_\e}
\int_{\e^r(\zeta+J)}\int_{\e(\xi+\square)}
 \left| \anop\bigl(\bigl(\partial_{x_k}v_{0}-
\anop(\partial_{x_k}v_{0})\bigl)\nabla_y \Phi_k\bigl)\right|^2\, dxdt\\
&\quad \le
\int_0^T\int_{\Omega}\anop(\left|\bigl(\partial_{x_k}v_{0}-
\anop(\partial_{x_k}v_{0})\bigl) \nabla_y \Phi_k\right|^2)\, dxdt\\
&\quad =
\int_0^T\int_{\Omega}\int_0^1\int_{\square} \left|\bigl(\partial_{x_k}v_{0}-
\anop(\partial_{x_k}v_{0})\bigl)\nabla_y \Phi_k\right|^2\ufnop(1)\, dydsdxdt\\
&\quad =
\int_0^T\int_{\Omega}\|\nabla_y \Phi_k(x,t)\sqrt{\ufnop(1)}\|_{L^2(\square\times J)}^2 \bigl(\partial_{x_k}v_{0}-
\anop(\partial_{x_k}v_{0})\bigl)^2\, dxdt\\
&\quad \le
\|\nabla_y\Phi_k\|_{L^{\infty}(\Omega\times I;L^2(\square\times J))}^2\bigl\|\partial_{x_k}v_{0}-
\anop(\partial_{x_k}v_{0})\bigl\|_{L^2(\Omega\times I)}^2
\to 0\quad \text{ as }\ \e_n\to 0_+.
\end{align*}
Here we used the facts $|\ufnop(1)|\le 1$, \eqref{I2est} and \eqref{phireg} in the last line. Thus we obtain \eqref{final}.

Combining \eqref{I1est}, \eqref{I2est} and \eqref{I3est}, we obtain
$$
\bigl\|\nabla v_{\e_n}-\nabla v_0-\sum_{k=1}^N\anop(\partial_{x_k}v_0)\,\anop(\nabla_y \Phi_k)\bigl\|_{L^2(\Omega\times I)}
\le I_1^{\e_n}+ I_2^{\e_n}+ I_3^{\e_n}\to 0 \ \text{ as } \e_n\to 0_+,
$$
which completes the proof of Theorem \ref{T:cor}.

\subsection{Proof of Corollary \ref{C:cor}}
Let $j_{\e_n} := a_{\e_n} \nabla v_{\e_n}$ be the diffusion flux of \eqref{eq:P}.
Thanks to Theorem \ref{T:cor}, it follows that
\begin{align*}
\lefteqn{
\bigl\|
j_{\e_n} - a_{\e_n} \bigl( \nabla v_{0}+\sum_{k=1}^N\,\anop( \partial_{x_k} v_0)\,\anop(\nabla_y \Phi_k)\bigl)
\bigl\|_{L^2(\Omega \times I)}
}\\
&\quad \leq  \bigl\|
\nabla v_{\e_n}- \nabla v_{0} - \sum_{k=1}^N\,\anop( \partial_{x_k} v_0)\,\anop(\nabla_y \Phi_k)
\bigl\|_{L^2(\Omega \times I)} \to 0 \quad\text{ as } \e_n\to 0_+.
\end{align*}
Therefore, although $j_{\e_n} \to j_{\rm hom}$ weakly in $[L^2(\Omega \times I)]^N$ as $\e_n \to 0_+$ (see Theorem \ref{T:1}), one can derive that
\begin{align}\label{eq:crj}
\lim_{\e_n \to 0_+} \bigl\|
j_{\e_n} - j_{\rm hom} - \bigl[ a_{\e_n} \bigl( \nabla v_{0}+\sum_{k=1}^N\,\anop( \partial_{x_k} v_0)\,\anop(\nabla_y \Phi_k) \bigl) - j_{\rm hom} \bigl]
\bigl\|_{L^2(\Omega \times I)} = 0.
\end{align}

Furthermore, noting that
\begin{equation}\label{eq:div-H-1}
\|\mathrm{div} \,\varphi\|_{H^{-1}(\Omega)} = \sup_{\|\nabla w\|_{L^2(\Omega)} = 1} \int_\Omega \varphi \cdot \nabla w \, dx \leq \|\varphi\|_{L^2(\Omega)}
\quad \mbox{ for } \ \varphi \in [L^2(\Omega)]^N,
\end{equation}
we conclude that
\begin{align*}
\lefteqn{
 \bigl\| \partial_t v_{\e_n}^{1/p} - \partial_t v_0^{1/p} - \mathrm{div} \bigl[ a_{\e_n} \Bigl( \nabla v_{0}+\sum_{k=1}^N\,\anop( \partial_{x_k} v_0)\,\anop(\nabla_y \Phi_k) \Bigl) - j_{\rm hom} \bigl] \bigl\|_{L^2(I;H^{-1}(\Omega))}
}\\
&\stackrel{\eqref{eq:P}, \eqref{eq:P0},\eqref{eq:div-H-1}}{\leq} \ \bigl\| j_{\e_n} - j_{\rm hom} - \bigl[ a_{\e_n} \Bigl( \nabla v_{0}+\sum_{k=1}^N\,\anop( \partial_{x_k} v_0)\,\anop(\nabla_y \Phi_k) \Bigl) - j_{\rm hom} \bigl] \bigl\|_{L^2(\Omega\times I)}\\
&\qquad +
\|f_{\e_n}- f\|_{L^2(I;H^{-1}(\Omega))}
\stackrel{\eqref{eq:crj}}{\to} 0
\quad\text{ as } \e_n \to 0_+.
\end{align*}
This completes the proof.

\begin{rmk}\label{R:nvanish}
\rm
We stress that although the corrector term of $j_{\e_n}$ converges to zero weakly in $[L^2(\Omega\times I)]^N$, it does not converge strongly in general. Indeed, suppose that $a(y,s)$ is smooth and $r\neq 2$. Then we observe that
\begin{align}
&a_{\e_n} \Bigl( \nabla v_{0}+\sum_{k=1}^N\,\anop( \partial_{x_k} v_0)\,\anop(\nabla_y \Phi_k) \Bigl) - j_{\rm hom}\label{nvanish}\\
&=
\{a_{\e_n}( \nabla v_{0}+\nabla_y z(x,t,\mnx,\mnt)) - j_{\rm hom}\}
-
\{a_{\e_n}(\nabla_y z(x,t,\mnx,\mnt)-\, \anop (\nabla_yz))\}\nonumber\\
&\quad
- \underbrace{a_{\e_n}\sum_{k=1}^N \Bigl(
\anop\bigl( (\partial_{x_k} v_0)\nabla_y\Phi_k\bigl)- \,\anop( \partial_{x_k} v_0)\,\anop(\nabla_y \Phi_k)  \Bigl)}_{\to\ 0 \text{ strongly in } [L^2(\Omega\times I)]^N \text{ by \eqref{final}}}.\nonumber
\end{align}
Moreover, as for the second term in \eqref{nvanish}, noting by the $(\e_n\square\times \e_n^rJ)$-periodicity of $\Phi_k(\mnx,\mnt)$ that
\begin{align*}
\anop (\nabla_yz)
&=
\sum_{k=1}^N
\left(\int_0^1\int_\square \partial_{x_k}v_0(\e_n\lfloor \mnx\rfloor+\e_n\sigma,\e_n^r\lfloor \mnt\rfloor+\e_b^r\rho)\, d\sigma d\rho\right) \nabla_y\Phi_k(\{\mnx\},\{\mnt\})\\
&=
\sum_{k=1}^N \anop (\partial_{x_k}v_0)\nabla_y\Phi_k(\mnx,\mnt),
\end{align*}
we see by the $(\square\times J)$-periodicity of $\Phi_k$ and \cite{AGK} (i.e.,  $\nabla_y\Phi_k\in [L^{\infty}(\square\times J)]^N$) that
\begin{align*}
\lefteqn{\|a_{\e_n}(\nabla_y z(x,t,\mnx,\mnt)-\, \anop (\nabla_yz))\|_{L^2(\Omega\times I)}}\\
&\le \sum_{k=1}^N \|\nabla_y\Phi_k(\mnx,\mnt)\|_{L^{\infty}(\Omega\times I)}
\|\partial_{x_k}v_0- \anop(\partial_{x_k}v_0)\|_{L^2(\Omega\times I)}\\
&= \sum_{k=1}^N \|\nabla_y\Phi_k\|_{L^{\infty}(\square\times J)}
\underbrace{\|\partial_{x_k}v_0- \anop(\partial_{x_k}v_0)\|_{L^2(\Omega\times I)}}_{\to\ 0 \text{ by \eqref{I2est}} } \to 0\ \text{ as } \e_n\to 0_+.
\end{align*}
However, the first term in \eqref{nvanish} does not strongly converge to zero in general, i.e.,  
$$
a_{\e_n}(\nabla v_0+\nabla_y z(x,t,\mnx,\mnt))-j_{\rm hom} \not\to 0\quad \text{ strongly in }
[L^2(\Omega\times I)]^N
$$
(see \cite[Remark 6.2]{AO1} for more details).
\end{rmk}

\section*{Acknowledgment}
The author is partially supported by Division for Interdisciplinary Advanced Research and Education, Tohoku University and Grant-in-Aid for JSPS Fellows (No.~JP20J10143). He would like to thank Professor Goro Akagi (Tohoku University) who is his supervisor, for many stimulating discussions.

\end{document}